\documentclass[11pt]{article}

\usepackage{cite}
\usepackage{subfigure}
\usepackage{graphicx, color, graphpap}
\usepackage[normalem]{ulem}
\usepackage{amsmath,amssymb,amsthm}
\usepackage{ifthen}
\usepackage{mathrsfs}
\usepackage{mathtools}
\usepackage{enumitem}
\usepackage{lineno}
\usepackage{float}
\usepackage{fancyhdr}
\usepackage[us,12hr]{datetime} 
\usepackage{exscale}
\usepackage{tabularx}
\usepackage{latexsym}
\usepackage{fullpage}       
\usepackage{float}
\usepackage{verbatim}
\usepackage{multirow}
\usepackage{overpic}
\usepackage{comment} 
\usepackage{hyperref}
\usepackage{blkarray}
\usepackage{adjustbox}
\usepackage{tcolorbox}

\usepackage{pgfplots}
\usepackage{marvosym}
\usepackage[noabbrev]{cleveref}
\usepackage{tikz}
\usepackage{pgfplots}

\usepackage{algorithm}
\usepackage{algpseudocode}

\pgfplotsset{compat=1.11}
\usetikzlibrary{arrows, arrows.meta}
\usetikzlibrary{shapes.geometric, arrows.meta, positioning}

\tikzstyle{startstop} = [rectangle, rounded corners, minimum width=3.5cm, minimum height=1cm,text centered, draw=black, fill=blue!10]
\tikzstyle{process} = [rectangle, minimum width=4cm, minimum height=1cm, text centered, draw=black, fill=green!10]
\tikzstyle{io} = [trapezium, trapezium left angle=70, trapezium right angle=110, minimum width=3.5cm, minimum height=1cm, text centered, draw=black, fill=orange!15]
\tikzstyle{arrow} = [thick,->,>=stealth]

\definecolor{methD}{RGB}{230,120,50}
\definecolor{methDD}{RGB}{50,140,90}
\definecolor{methSDD}{RGB}{50,90,170}
\definecolor{methPSD}{RGB}{180,40,40}
\tikzset{
	visbox/.style={rounded corners=8pt, line width=1.2pt, align=center, inner sep=8pt},
	visnestouter/.style={visbox, minimum width=9.6cm, minimum height=4.4cm},
	visnest/.style={visbox, minimum width=7.6cm, minimum height=3.25cm},
	visnestmid/.style={visbox, minimum width=5.7cm, minimum height=2.2cm},
	visnestinn/.style={visbox, minimum width=3.4cm, minimum height=1.05cm,
		inner sep=3pt, font=\scriptsize\bfseries, align=center, text=methD},
	visnestpos/.style={anchor=south east},
	visnestinset/.style={xshift=-5pt, yshift=5pt},
}

\newcommand{\Diag}{\text{Diag}}
\newcommand{\diag}{\text{diag}}
\DeclareMathOperator{\ri}{ri}

\newcommand{\supp}{\text{supp}}

\newcommand{\A}{\mathcal{A}}

\DeclareMathOperator{\sd}{sd}

\newcommand{\proj}{\text{proj}}

\newtheorem{lem}{Lemma}[section]
\newtheorem{thm}{Theorem}[section]

\crefname{thm}{Theorem}{Theorems}
\Crefname{thm}{Theorem}{Theorems}
\crefname{problem}{Problem}{Theorems}
\Crefname{problem}{Problem}{Theorems}
\Crefname{assump}{Assumption}{Theorems}
\crefname{assump}{Assumption}{Theorems}
\crefname{assumption}{Assumption}{Assumptions}
\Crefname{assumption}{Assumption}{Assumptions}
\crefname{conjecture}{Conjecture}{Theorems}
\Crefname{conjecture}{Conjecture}{Theorems}
\crefname{prop}{Proposition}{Propositions}
\Crefname{prop}{Proposition}{Propositions}
\crefname{cor}{Corollary}{Corollaries}
\Crefname{cor}{Corollary}{Corollaries}
\crefname{lem}{Lemma}{Lemmas}
\Crefname{lem}{Lemma}{Lemmas}
\theoremstyle{definition}
\crefname{conj}{Conjecture}{Conjectures}
\Crefname{conj}{Conjecture}{Conjectures}
\crefname{remark}{Remark}{Remarks}
\Crefname{remark}{Remark}{Remarks}
\crefname{rmk}{Remark}{Remarks}
\Crefname{rmk}{Remark}{Remarks}
\crefname{example}{Example}{Examples}
\Crefname{example}{Example}{Examples}
\crefname{align}{}{}
\Crefname{align}{}{}
\crefname{equation}{}{}
\Crefname{equation}{}{}

\def\eqref#1{{\normalfont(\ref{#1})}}

\usepackage{authblk}

\author[1]{Hao Hu\footnote{School of Mathematical and Statistical Sciences, Clemson University, Clemson, USA; Email: \url{hhu2@clemson.edu}; Research supported by the Air Force Office of Scientific Research under award number FA9550-23-1-0508.}}

\begin{document}

\title{On the Singularity Degree of Shor Relaxations for \(0\)–\(1\) Programs}

\break
\date{\today}
\maketitle

\medskip


\begin{abstract}
Shor relaxations are standard tools for nonconvex quadratic optimization, but
they may fail Slater's condition.  The
singularity degree quantifies this degeneracy by the number of facial-reduction
steps needed to reach the minimal face.  We determine how the constraints
defining a binary feasible set affect this quantity.  For every
nonempty binary set defined by linear equations and inequalities, the
Shor relaxation has singularity degree at most one.  It is zero precisely when
the linear programming (LP) relaxation contains a point in the open cube that
strictly satisfies every inequality, and is one otherwise.  As an algorithmic
consequence of this characterization, Slater's condition can be restored by
solving an LP rather than an auxiliary semidefinite program (SDP).  In
sharp contrast, for every $n\geq1$, we construct a binary feasible set defined
by quadratic equalities whose Shor relaxation has a positive semidefinite
matrix variable of order $n+1$ and singularity degree $n$.  Thus, singularity
degree is uniformly bounded in the linearly constrained case but can grow
linearly with the matrix order when quadratic defining constraints are used.
\end{abstract}


\section{Introduction}

The Shor relaxation is a standard convexification for binary
quadratic optimization.  Its facial geometry depends on the structure retained
by the lifting step.  The exact rank-one lift of $x\in\{0,1\}^n$ is
$Y:=\binom{1}{x}\binom{1}{x}^T$, whose lower-right block is $xx^T$.
Binarity gives $(xx^T)_{ii}=x_i$.  The Shor relaxation drops the rank-one
condition by replacing $xx^T$ with a symmetric matrix variable
$X\in\mathbb{S}^n$ and
imposing $Y\succeq0$, while retaining the normalization constraint $Y_{00}=1$
and the arrow constraints $X_{ii}=x_i$, $i=1,\ldots,n$.  Constraints from the
original formulation are imposed in addition: linear constraints remain
linear, while quadratic constraints are linearized as affine constraints in
the lifted variables.  This gives the standard SDP relaxation for quadratic
binary optimization
\cite{shor1987quadratic,poljak1995recipe,rendl2010integer}.  How the defining
constraints enter this lifted system is central to our results.

A central regularity question for an SDP is whether its feasible set satisfies
Slater's condition.  When strict feasibility fails, the feasible set lies in a
proper face of the ambient cone.  Facial reduction replaces that cone by a
smaller face containing the feasible set
\cite{borwein1981facial,waki2013facial,pataki2013strong}.  Partial
polyhedrality yields refined bounds on the number of facial-reduction steps
\cite{lourencoMuramatsuTsuchiya2018partial}.  The minimum number of
reductions needed to reach the minimal face is the singularity degree.  This
quantity governs error bounds for semidefinite systems
\cite{sturm2000error,sremac2021error} and is also linked to exponential
solution size in SDP \cite{patakiTouzov2024exponential}.  For the Shor
relaxation studied here, Slater's condition and singularity degree depend only
on the feasible conic system and are independent of the objective function.
Accordingly, our results are stated in terms of feasible systems and the
constraints defining the underlying binary sets.  Throughout the paper,
``Shor relaxation'' refers to this primal semidefinite feasible system.

Our first result gives a complete characterization for linearly constrained
binary sets.  For every nonempty binary set defined by linear equations and
inequalities, its Shor relaxation has singularity degree zero precisely
when the linear programming relaxation contains a point in the open cube that
strictly satisfies every inequality; otherwise, its singularity degree is one.
Thus, the singularity degree is always at most one, and this bound is attained.
The proof combines the geometry of the arrow-constrained spectrahedron
\cite{gouveia2010theta,laurent2005semidefinite} with the strict-complementarity
theorem of Goldman and Tucker
\cite{goldman1956theory}: the former yields the Slater criterion, while the
latter produces the certificate used to expose the minimal face in one step.
A relative-interior point of the associated polyhedral slice, used to verify
that the exposed face is minimal, can also be found by linear programming.
This is standard for polyhedra; see
\cite[Corollary~2A]{goldman1956theory} and
\cite[Exercise~3.27]{bertsimas1997introduction}.  Thus,
Slater's condition can be restored without solving an auxiliary SDP.

Our second result shows that the restriction to linear constraints is
essential.  For every $n\geq1$, we construct a binary feasible set defined by
quadratic equalities whose Shor relaxation has a positive semidefinite matrix
variable of order $n+1$ and singularity degree $n$.  Thus, the singularity
degree can grow linearly with the matrix order even in the presence of the
arrow constraints.  The construction is formulation-sensitive: the underlying
binary set is the singleton $\{0\}$, but the quadratic formulation produces
high singularity degree, whereas the linear formulation $x=0$ has singularity
degree one.  The proof rotates Sturm's example
\cite[Example~2]{sturm2000error} using a Vandermonde transformation
so that each facial-reduction step removes exactly one coordinate.

Together, the two results identify a sharp boundary for facial degeneracy in
Shor relaxations.  With linear defining constraints, the Shor relaxation
is either strictly feasible or reaches its minimal face in one
facial-reduction step.  Once quadratic defining constraints are allowed, the
number of required steps can instead grow linearly with the order of the
lifted matrix.  Thus, the one-step behavior is specific to the linearly
constrained case and does not follow from the arrow constraints alone.

\paragraph{Relation to previous work.}
Tun{\c{c}}el \cite{tuncel2001slater} related dimension and
affine-hull information for a nonconvex feasible set to the construction of
Slater points for a broad class of SDP relaxations.  Hu
and Li \cite[Rem.~3.2]{hu2023facial} developed LP-based partial facial
reduction for Shor relaxations of quadratically constrained quadratic
programs.  Hu and Yang \cite[Lemma~4.4]{hu2024affine} studied a partial
facial-reduction procedure based on the diagonally dominant cone that
terminates after at most one step, and also proposed affine FR.  Hu and Xu
\cite{hu2025primal} proposed a primal facial-reduction method that uses
feasible points of the underlying combinatorial problem to simplify the
reduction steps.  Our results complement these approaches by identifying the
minimal face containing the entire Shor relaxation and determining its
standard singularity degree exactly.

\paragraph{Organization.}
\Cref{sec_prel} reviews the first-level binary lift, facial reduction, and the
result of Goldman and Tucker used later.
\Cref{sec:linear_binary_shor_degree}
proves the exact singularity-degree dichotomy between zero and one for linear
constraints, characterizes the minimal face, and develops an LP-based procedure
for restoring Slater's condition.  \Cref{sec:binary_high_singularity_degree}
explains why a direct binary transcription of Sturm's example has singularity
degree one and presents a quadratic construction attaining singularity degree
$n$.

\section{Preliminaries}\label{sec_prel}

\subsection{Notation}

For a positive integer $k$, let $\mathbb{R}^k_+$ and
$\mathbb{R}^k_{++}$ denote the nonnegative and strictly positive orthants,
respectively.  We write $\mathbb{S}^k$ for the space of real symmetric
$k\times k$ matrices, equipped with the trace inner product
\[
\langle U,V\rangle:=\operatorname{trace}(UV).
\]
The positive semidefinite (PSD) and positive-definite cones are denoted by
$\mathbb{S}^k_+$ and $\mathbb{S}^k_{++}$; we write $U\succeq0$ and
$U\succ0$ for membership in these cones.  For $U\in\mathbb{S}^k$,
$\diag(U)$ is its diagonal vector, and for $u\in\mathbb{R}^k$, $\Diag(u)$ is
the diagonal matrix with diagonal $u$.  The all-ones vector is denoted by
$\mathbf{1}$, with its dimension determined by context.  For a convex set
$\mathcal C$, $\ri(\mathcal C)$ denotes its relative interior.  For an
element $u$ of an inner-product space, we write
$u^\perp:=\{v\mid\langle u,v\rangle=0\}$; for a vector $u$, $\supp(u)$ is the
set of indices of its nonzero components.

For each $k\geq0$, we index the coordinates of $\mathbb{R}^{k+1}$, and hence
the rows and columns of matrices in $\mathbb{S}^{k+1}$, by
$\{0,1,\ldots,k\}$.  Let $e_0,\ldots,e_k$ denote the corresponding standard
unit vectors.  We use the symmetrized matrix units
\begin{equation}\label{eq:symmetric_basis_matrices}
E_{ii}:=e_ie_i^T,
\qquad
E_{ij}:=\frac12(e_ie_j^T+e_je_i^T)\quad(i\neq j),
\end{equation}
so that $\langle E_{ij},Y\rangle=Y_{ij}$ for $Y\in\mathbb{S}^{k+1}$.
The value of $k$, and hence the dimensions of $e_i$ and $E_{ij}$, will always
be clear from context.

\subsection{Shor relaxations for binary programs}
\label{subsec:shor_relaxations_binary_programs}
Let $A\in\mathbb{R}^{p\times n}$, $C\in\mathbb{R}^{q\times n}$,
$b\in\mathbb{R}^p$, and $d\in\mathbb{R}^q$.  Let $m$ be a nonnegative
integer.  For $j=1,\ldots,m$, let $Q_j\in\mathbb{S}^n$,
$q_j\in\mathbb{R}^n$, and $r_j\in\mathbb{R}$.
Consider the nonempty binary feasible set
\[
\mathcal X
:=
\left\{
x\in\{0,1\}^n
\ \middle|\
\begin{aligned}
Ax&=b, & Cx&\leq d,\\
x^TQ_jx+2q_j^Tx&=r_j &&(j=1,\ldots,m)
\end{aligned}
\right\}.
\]

The terminology used below refers to the displayed formulation.  The binary
domain $x\in\{0,1\}^n$ already encodes the identities $x_i^2=x_i$; we do not
count these identities as quadratic defining constraints.  We call the set
\emph{linearly constrained} when its formulation contains only the linear
constraints $Ax=b$ and $Cx\leq d$, corresponding to $m=0$.  We call it
\emph{quadratically constrained} when the formulation includes additional
explicit quadratic constraints, such as those indexed by
$j=1,\ldots,m$ above.  Thus, this distinction concerns the chosen
formulation, not polynomial identities that hold automatically on the binary
domain.

To construct the Shor relaxation \cite{shor1987quadratic} for $\mathcal X$,
introduce the lifted matrix variable $Y\in\mathbb{S}^{n+1}$ and write it in
block form as
$Y=\left(\begin{smallmatrix}Y_{00}&x^T\\x&X\end{smallmatrix}\right)$, where
$x\in\mathbb{R}^n$ and $X\in\mathbb{S}^n$.  In the relaxation below, $x$
and $X$ denote the blocks of $Y$ in this decomposition.
The rows and columns of $Y$ are indexed by $\{0,1,\ldots,n\}$, with index $0$
corresponding to the constant term.  The exact lift satisfies
$Y=\binom{1}{x}\binom{1}{x}^T$ and hence $X=xx^T$.  Replacing this rank-one
representation by positive semidefiniteness and linearizing the quadratic
constraints gives
\[
\mathcal S_{\mathrm{Shor}}
:=
\left\{
Y\in\mathbb{S}^{n+1}_+
\ \middle|\
\begin{aligned}
Y_{00}&=1, & \diag(X)&=x,\\
Ax&=b, & Cx&\leq d,\\
\langle Q_j,X\rangle+2q_j^Tx&=r_j &&(j=1,\ldots,m)
\end{aligned}
\right\}.
\]
Thus, the original linear constraints are retained, while each quadratic form
is linearized by replacing $xx^T$ with $X$.  Quadratic inequalities can be
included similarly by imposing their linearizations.  We omit them because
the first result concerns the case without quadratic constraints and the
high-degree construction requires only quadratic equalities.  The displayed
constraint description is part of the formulation: different descriptions of
the same binary set may produce different Shor relaxations.

The condition $Y_{00}=1$ is the normalization constraint, and the
\emph{arrow constraints} are $\diag(X)=x$.  Equivalently, in full-matrix
coordinates, the arrow constraints are
\begin{equation}\label{eq:binary_arrow_constraints}
Y_{ii}=Y_{0i},\qquad i=1,\ldots,n.
\end{equation}
The equality $X_{ii}=x_i$, equivalently $Y_{ii}=Y_{0i}$, represents the
binary identity $x_i^2=x_i$.
With the convention \eqref{eq:symmetric_basis_matrices}, we write the
normalization and arrow constraints as
\begin{equation}\label{eq:arrow_matrix_defs}
\langle A_0,Y\rangle=1,\qquad
\langle A_i,Y\rangle=0\quad (i=1,\ldots,n),
\qquad
A_0 := E_{00}, \quad A_i := E_{ii}-E_{0i}.
\end{equation}
Throughout the paper, binary feasible sets are assumed to be nonempty unless
explicitly stated otherwise.

\subsection{Facial reduction and singularity degree}
\label{subsec:facial_reduction_algorithm}
Facial reduction was introduced by Borwein and Wolkowicz
\cite{borwein1981facial}; see Drusvyatskiy and Wolkowicz
\cite{drusvyatskiy2017many} for a broader treatment.  We briefly recall the
framework used below.
Consider the nonempty conic feasible set
\[
\mathcal F:=\{z\in K\mid \A(z)=b\},
\]
where $\A$ is linear, $\A^*$ denotes its adjoint, and $K$ is a closed,
full-dimensional, self-dual cone.
The cones used below, namely PSD cones, nonnegative orthants, and their
products, have these properties.
Slater's condition holds when $\mathcal F\cap\ri(K)\neq\emptyset$.

If Slater's condition fails, a facial-reduction step finds a multiplier $y$
such that
\[
w:=\A^*(y)\in K\setminus\{0\},
\qquad b^Ty=0.
\]
For every $z\in\mathcal F$,
$\langle w,z\rangle=\langle y,\A(z)\rangle=b^Ty=0$.
Thus, $w$ exposes the proper face $K\cap w^\perp$ containing $\mathcal F$.
If this is not yet the minimal face containing $\mathcal F$, facial reduction
is repeated within the current face.  The smallest number of steps needed to
reach the minimal face is the \emph{singularity degree}, denoted by
$\sd(\mathcal F)$; it is zero exactly when Slater's condition holds.

Specializing to an SDP, the first-step facial-reduction auxiliary system for
$\A(X)=b$, $X\succeq0$, is therefore
\begin{equation}\label{eq:first_fr_auxiliary}
\A^{*}(y)\succeq0,\qquad \A^{*}(y)\neq0, \qquad b^Ty=0.
\end{equation}
If a face $F$ contains $\mathcal F$, then finding a point in
$\mathcal F\cap\ri(F)$ shows that $F$ is the minimal face containing
$\mathcal F$.

Singularity degree can be arbitrarily large in SDPs.  A reindexed form of a
classical example due to Sturm \cite[Example~2]{sturm2000error} (see also
Tun{\c{c}}el \cite[Sec.~2.6]{tuncel2010polyhedral}) is
\begin{equation}\label{eq:sturm_example}
\left\{
Z\in\mathbb{S}^{n+1}_+
\;\middle|\;
Z_{00}=1,\quad
Z_{nn}=0,\quad
Z_{ii}-Z_{0,i+1}=0,\ i=1,\ldots,n-1
\right\}.
\end{equation}
This system has singularity degree $n$: facial reduction forces the diagonal
entries $Z_{nn},\allowbreak Z_{n-1,n-1},\ldots,\allowbreak Z_{11}$ to vanish
sequentially.

We also use the standard description of PSD faces: every face of
$\mathbb{S}^n_+$ can be written, for some full-column-rank matrix
$V\in\mathbb{R}^{n\times r}$, as
\[
F_V:=\{VRV^T\mid R\succeq0\},
\qquad
\ri(F_V)=\{VRV^T\mid R\succ0\}.
\]

Facial reduction for linear programming terminates after at most one step.
Indeed, for a nonempty linear programming feasible system, the singularity
degree relative to its polyhedral cone is zero when Slater's condition holds
and one otherwise.  We use the following coordinate form of this fact for the
nonnegative orthant.  It is a standard consequence of the
strict-complementarity theorem of Goldman and Tucker
\cite{goldman1956theory}.

\begin{lem}
\label{lem:nonnegative_orthant_minimal_face}
Let $M\in\mathbb{R}^{m\times N}$ and $g\in\mathbb{R}^m$, and define
\[
\begin{aligned}
\mathcal U&:=\{z\in\mathbb{R}^N_+\mid Mz=g\},\\
\mathcal Z&:=\{j\in\{1,\ldots,N\}\mid
z_j=0\ \text{for every }z\in\mathcal U\}.
\end{aligned}
\]
Suppose that $\mathcal U$ is nonempty.  Then there exists
$y\in\mathbb{R}^m$ such that
\[
h:=M^Ty\geq0,\qquad g^Ty=0,\qquad \supp(h)=\mathcal Z.
\]
Thus, when $\mathcal Z\neq\emptyset$, $h$ exposes the minimal face of
$\mathbb{R}^N_+$ containing $\mathcal U$; when $\mathcal Z=\emptyset$,
$\mathcal U\cap\mathbb{R}^N_{++}\neq\emptyset$.
\end{lem}

We will also use the fact that a relative-interior point of a polyhedral set
can be found by solving a linear program; see
\cite{mehdiloo2021finding} for an explicit construction.

\section{The Shor relaxation of a linearly constrained binary set}
\label{sec:linear_binary_shor_degree}

We show that the Shor relaxation of every nonempty linearly constrained binary
set has singularity degree at most one.
Fix matrices $A\in\mathbb{R}^{p\times n}$ and
$C\in\mathbb{R}^{q\times n}$, and vectors $b\in\mathbb{R}^p$ and
$d\in\mathbb{R}^q$.  Consider the binary feasible set
\[
\mathcal X:=\{x\in\{0,1\}^n\mid Ax=b,\ Cx\leq d\},
\]
which we assume is nonempty.  Its linear programming (LP) relaxation is
\begin{equation}\label{eq:shor_lp_relaxation}
P
:=
\{x\in[0,1]^n\mid Ax=b,\ Cx\leq d\}.
\end{equation}

The Shor relaxation retains the original linear constraints and imposes the
normalization constraint, the arrow constraints, and the PSD constraint.
After introducing a nonnegative slack vector
$s\in\mathbb{R}^q_+$ for $Cx\leq d$, the variables $(Y,s)$ lie in the
product cone $\mathbb{S}^{n+1}_+\times\mathbb{R}^q_+$.  Using the block
convention $Y=\left(\begin{smallmatrix}1&x^T\\x&X\end{smallmatrix}\right)$,
the feasible set is
\begin{equation}\label{eq:shor_linear_binary_sdp}
\begin{aligned}
S
:=\bigl\{(Y,s)\in\mathbb{S}^{n+1}_+\times\mathbb{R}^q_+
\ \big|\ {}
&Y_{00}=1,\\
&\diag(X)=x,\\
&Ax=b,\\
&Cx+s=d
\bigr\}.
\end{aligned}
\end{equation}
For completeness, a quadratic objective $x^THx+c^Tx$, where
$H\in\mathbb{S}^n$ and $c\in\mathbb{R}^n$, is replaced in the Shor relaxation
by $\langle H,X\rangle+c^Tx$.  This does not change the feasible
set $S$, and all facial-reduction statements below concern $S$ alone.

For later reference, we state the first-step facial-reduction auxiliary system
for \eqref{eq:shor_linear_binary_sdp} explicitly.  Let
$\lambda\in\mathbb{R}$, $r\in\mathbb{R}^n$, $\mu\in\mathbb{R}^p$, and
$\nu\in\mathbb{R}^q$ be multipliers for the normalization, the arrow
constraints, $Ax=b$, and $Cx+s=d$, respectively.  Since
$\mathbb{S}^{n+1}_+\times\mathbb{R}^q_+$ is self-dual, the system is
\begin{equation}\label{eq:shor_fr_auxiliary}
\begin{aligned}
W
&:=\lambda E_{00}
+\sum_{i=1}^n r_i(E_{ii}-E_{0i})
+\sum_{i=1}^n(A^T\mu+C^T\nu)_iE_{0i}\succeq0,\\
&\nu\geq0,\qquad (W,\nu)\neq(0,0),\\
&\lambda+b^T\mu+d^T\nu=0.
\end{aligned}
\end{equation}
For every $x\in\mathcal X$, the rank-one matrix
\[
\begin{pmatrix}1\\x\end{pmatrix}
\begin{pmatrix}1\\x\end{pmatrix}^T
\]
together with $s=d-Cx$ is feasible in
\eqref{eq:shor_linear_binary_sdp}.  Hence
$S$ is an SDP relaxation of $\mathcal X$.
Here ``Shor relaxation'' refers exactly to
\eqref{eq:shor_linear_binary_sdp}.  No strengthening constraints on $Y$,
such as RLT constraints or linearizations of additional valid quadratic
constraints, are imposed.  We view each scalar slack as a $1\times1$ PSD block
and compute singularity degree with respect to this product cone.

The Shor relaxation \eqref{eq:shor_linear_binary_sdp} is standard; see
\cite{poljak1995recipe}, \cite[Sec.~3.1, Eq.~(9)]{rendl2010integer}, and
\cite[Sec.~2.2]{galli2014compact}.  The unit-cube projection used below is also
standard; see \cite[Sec.~3.1]{gouveia2010theta} and
\cite[Sec.~4.2, Eq.~(63) and Lem.~10]{laurent2005semidefinite}.
For later use, define
\[
B_n
:=
\left\{
Y=
\begin{pmatrix}1&x^T\\x&X\end{pmatrix}
\in\mathbb{S}^{n+1}_+
\ \middle|\
\diag(X)=x
\right\}.
\]
Here $\proj_x$ denotes projection onto the $x$-block.  The standard
projection identity is
\begin{equation}\label{eq:arrow_projection_cube}
\proj_x(B_n)=[0,1]^n.
\end{equation}
An explicit lift is obtained by setting
\[
D(x):=\Diag\bigl(x_1(1-x_1),\ldots,x_n(1-x_n)\bigr)
\]
and
\begin{equation}\label{eq:shor_explicit_lift}
Y(x)
:=
\begin{pmatrix}
1&x^T\\
x&xx^T+D(x)
\end{pmatrix}.
\end{equation}
For every $x\in[0,1]^n$, the lower-right block has diagonal $x$ and the
Schur complement of the $(0,0)$ entry is $D(x)\succeq0$, so $Y(x)\in B_n$.
The additional point needed for our facial analysis is to determine which
points in this projection admit a positive-definite lift.

\begin{lem}[Positive-definite lifts]
\label{lem:arrow_positive_definite_lift}
For a given $x\in[0,1]^n$, there exists
$X\in\mathbb{S}^n$ such that
$\left(\begin{smallmatrix}1&x^T\\x&X\end{smallmatrix}\right)\in
B_n\cap\mathbb{S}^{n+1}_{++}$ if and only if $0<x<\mathbf{1}$.
\end{lem}

\begin{proof}
If $Y=\left(\begin{smallmatrix}1&x^T\\x&X\end{smallmatrix}\right)\in
B_n\cap\mathbb{S}^{n+1}_{++}$, then each
principal submatrix
\[
\begin{pmatrix}
1&x_i\\
x_i&x_i
\end{pmatrix}
\]
is positive definite, and hence $0<x_i<1$.

Conversely, suppose that $0<x<\mathbf{1}$.  Then $D(x)\succ0$, so the Schur
complement of the $(0,0)$ entry in \eqref{eq:shor_explicit_lift} is
positive definite.  Hence $Y(x)\succ0$.
\end{proof}

It follows immediately from \eqref{eq:arrow_projection_cube} that
\begin{equation}\label{eq:shor_projection_P}
\proj_x(S)=P.
\end{equation}
Indeed, every $x\in P$ admits the lift $Y(x)$ in
\eqref{eq:shor_explicit_lift}, with slack $s=d-Cx$, while every feasible
$(Y,s)$ projects into $P$ by the defining constraints of $S$.  Consequently,
the Shor relaxation and the LP relaxation have the same optimal value
for every linear objective in $x$.

We can now characterize Slater's condition and the singularity degree of $S$.
\begin{thm}\label{thm:shor_linear_sd_bound}
The Shor relaxation \eqref{eq:shor_linear_binary_sdp} has
singularity degree at most one.  More precisely,
\[
\sd(S)
=
\begin{cases}
0,
&\text{if there exists }x\in P\text{ such that }
0<x<\mathbf{1}\text{ and }Cx<d,\\
1,
&\text{otherwise}.
\end{cases}
\]
\end{thm}

\begin{proof}
By \eqref{eq:shor_projection_P}, $\proj_x(S)=P$.
\Cref{lem:arrow_positive_definite_lift} shows that $Y$ can be positive definite exactly when
$0<x<\mathbf{1}$, while the slack vector lies in $\mathbb{R}^q_{++}$ exactly
when $Cx<d$.  Hence $S$ satisfies Slater's condition if and
only if the first case in the theorem holds.

It remains to show that one facial-reduction step suffices when Slater's
condition fails.  Consider the following affine slice of the nonnegative
orthant:
\begin{equation}\label{eq:shor_nonnegative_orthant_slice}
\mathcal U
:=
\left\{
(u,v,s)\in\mathbb{R}^{2n+q}_+
\ \middle|\
u+v=\mathbf{1},\ Au=b,\ Cu+s=d
\right\}.
\end{equation}
This set is nonempty because it contains
$(x,\mathbf{1}-x,d-Cx)$ for every $x\in P$.  Write the equality system in
\eqref{eq:shor_nonnegative_orthant_slice} as $Mz=g$.  Since Slater's
condition fails, $\mathcal U\cap\mathbb{R}^{2n+q}_{++}=\emptyset$.  Hence
\Cref{lem:nonnegative_orthant_minimal_face} gives $\mathcal Z\neq\emptyset$ and
multipliers $(\delta,\mu,\nu)\in
\mathbb{R}^n\times\mathbb{R}^p\times\mathbb{R}^q$ such that
\begin{equation}\label{eq:shor_multiplier_identities}
\begin{gathered}
M^T(\delta,\mu,\nu)=(\alpha,\beta,\gamma)\geq0,
\qquad
\mathbf{1}^T\delta+b^T\mu+d^T\nu=0,\\
\alpha=\delta+A^T\mu+C^T\nu,
\qquad \beta=\delta,
\qquad \gamma=\nu.
\end{gathered}
\end{equation}
Moreover, $(\alpha,\beta,\gamma)$ is nonzero, and its positive entries
correspond exactly to components of
$(x,\mathbf{1}-x,s)$ that vanish throughout
\eqref{eq:shor_nonnegative_orthant_slice}.

Define
\begin{equation}\label{eq:shor_one_step_exposing_vector}
W
:=
\sum_{i=1}^n\alpha_i e_ie_i^T
+\sum_{i=1}^n\beta_i(e_0-e_i)(e_0-e_i)^T
\succeq0.
\end{equation}
For every $(Y,s)$ satisfying the affine equations in
\eqref{eq:shor_linear_binary_sdp}, the normalization and arrow
constraints give
\begin{align*}
\langle W,Y\rangle+\gamma^Ts
&=
\sum_{i=1}^n\alpha_iX_{ii}
+\sum_{i=1}^n\beta_i(1-2x_i+X_{ii})
+\gamma^Ts\\
&=\alpha^Tx+\beta^T(\mathbf{1}-x)+\gamma^Ts\\
&=0.
\end{align*}
If $\gamma\neq0$, then $(W,\gamma)\neq0$.  If $\gamma=0$, then
$(\alpha,\beta)\neq0$, and \eqref{eq:shor_one_step_exposing_vector}
is a nonzero nonnegative sum of rank-one PSD matrices.  Thus $(W,\gamma)$ is
nonzero.  By \eqref{eq:shor_multiplier_identities},
$A^T\mu+C^T\gamma=\alpha-\beta$, and hence
\begin{align*}
W
&=(\mathbf{1}^T\beta)E_{00}
+\sum_{i=1}^n(\alpha_i+\beta_i)(E_{ii}-E_{0i})
+\sum_{i=1}^n(A^T\mu+C^T\gamma)_iE_{0i},\\
0&=\mathbf{1}^T\beta+b^T\mu+d^T\gamma.
\end{align*}
Therefore, with $\lambda=\mathbf{1}^T\beta$, $r=\alpha+\beta$, and
$\nu=\gamma$, the pair $(W,\gamma)$ satisfies
\eqref{eq:shor_fr_auxiliary}.  Thus it is a nonzero exposing vector for
the product cone.

We finally verify that the face exposed by $(W,\gamma)$ is the minimal face.
Let
\begin{align*}
I_0&:=\{i\mid x_i=0\ \text{for every }x\in P\},\\
I_1&:=\{i\mid x_i=1\ \text{for every }x\in P\},\\
K_0&:=\{k\mid d_k-(Cx)_k=0\ \text{for every }x\in P\}.
\end{align*}
The support property in \Cref{lem:nonnegative_orthant_minimal_face} gives
$\alpha_i>0$ exactly on $I_0$, $\beta_i>0$ exactly on $I_1$, and
$\gamma_k>0$ exactly on $K_0$.  Hence the face of
$\mathbb{S}^{n+1}_+$ exposed by $W$ is
\[
F
=
\{Y\succeq0\mid Ye_i=0\ (i\in I_0),\quad
Y(e_0-e_i)=0\ (i\in I_1)\}.
\]
The face of the nonnegative orthant exposed by $\gamma$ is
\[
G:=\{s\in\mathbb{R}^q_+\mid \gamma^Ts=0\}.
\]
Therefore the face of $\mathbb{S}^{n+1}_+\times\mathbb{R}^q_+$ exposed by
$(W,\gamma)$ is $F\times G$.

Choose $(\bar x,\mathbf{1}-\bar x,\bar s)\in\ri(\mathcal U)$, and let
$J:=\{1,\ldots,n\}\setminus(I_0\cup I_1)$.  Then
$0<\bar x_j<1$ for every $j\in J$, and $\bar s_k>0$ for every
$k\notin K_0$.  Set
\[
v_0:=e_0+\sum_{i\in I_1}e_i,
\qquad
V:=\begin{pmatrix}v_0&(e_j)_{j\in J}\end{pmatrix},
\]
whose columns form a basis for the kernel of $W$,
and define
\[
R
:=
\begin{pmatrix}
1&\bar x_J^T\\
\bar x_J&\bar x_J\bar x_J^T+
\Diag\bigl((\bar x_j(1-\bar x_j))_{j\in J}\bigr)
\end{pmatrix}\succ0.
\]
Then $\bar Y:=VRV^T$ belongs to $\ri(F)$, satisfies the arrow constraints,
and its $x$-block is $\bar x$.  Moreover, $R_{00}=1$ gives
$\bar Y_{00}=1$, and membership of
$(\bar x,\mathbf{1}-\bar x,\bar s)$ in $\mathcal U$ gives
$A\bar x=b$ and $C\bar x+\bar s=d$.  Thus
$(\bar Y,\bar s)\in S$.  Since
$\bar s\in\ri(G)$, we have
\[
(\bar Y,\bar s)\in S\cap\ri(F\times G).
\]
Consequently, $F\times G$ is the minimal face of the product cone containing
$S$.  Since Slater's condition fails, the singularity degree
is at least one, while the single exposing vector $(W,\gamma)$ reaches the
minimal face.  Therefore $\sd(S)=1$.
\end{proof}

The bound in \Cref{thm:shor_linear_sd_bound} is attained.  For example,
the binary set
\[
\left\{x\in\{0,1\}^n\ \middle|\ \sum_{i=1}^n x_i=0\right\}
=\{0\}
\]
has a Shor relaxation that fails Slater's condition, and
$\sum_{i=1}^nE_{ii}$ is an exposing vector whose exposed face is the minimal
face.  Therefore its singularity degree is one, so the bound in
\Cref{thm:shor_linear_sd_bound} is tight.

The theorem has two useful consequences.  First, Slater's condition can be
restored by restricting the product cone to its minimal face, without solving
an auxiliary SDP.  A relative-interior point
$(\bar x,\mathbf{1}-\bar x,\bar s)$ of the polyhedral set $\mathcal U$ in
\eqref{eq:shor_nonnegative_orthant_slice} can be found by linear programming
\cite{mehdiloo2021finding}.  Its zero components identify the variables fixed
at zero or one and the inequalities that are always tight, and hence determine
the minimal face constructed in the proof of
\Cref{thm:shor_linear_sd_bound}.  The explicit matrix construction in that
proof then gives a feasible point in the relative interior of this face.
Thus, the minimal face and a relative-interior feasible point can be recovered
by solving an LP and applying explicit linear algebra.

Second, the bound $\sd(S)\leq1$ limits the severity of the degeneracy even if
facial reduction is not applied.  Sturm's H\"older error bound has exponent
$2^{-\sd(S)}$, so the exponent for the Shor relaxation considered here is at
least $1/2$.  Although $\sd(S)=1$ may still cause numerical difficulty, the
bound rules out the deeper degeneracy and progressively weaker error bounds
associated with long facial-reduction sequences
\cite{sturm2000error,sremac2021error}.

\section{High singularity degree in Shor relaxations of quadratically
constrained binary sets}
\label{sec:binary_high_singularity_degree}

\subsection{Sturm's example and the arrow constraints}

Section~\ref{sec:linear_binary_shor_degree} shows that the Shor
relaxation of every linearly constrained binary set has singularity
degree at most one.  We now show that this bound results from restricting the
defining constraints to be linear, rather than from the arrow constraints
alone.  We begin with a direct binary transcription of Sturm's example.  Its
interaction with the arrow constraints collapses the entire
facial-reduction sequence into one step, motivating the transformed
construction in the following subsection.

Recall Sturm's system \eqref{eq:sturm_example}.  Mimicking its chain with
binary variables gives
\begin{equation}\label{eq:boolean_sturm_attempt}
x_i^2-x_{i+1}=0,\quad i=1,\ldots,n-1,
\qquad x_n^2=0,
\qquad x\in\{0,1\}^n.
\end{equation}
These equations again force $x=0$, but their Shor relaxation does not inherit
the high singularity degree of \eqref{eq:sturm_example}.  Using the
normalization and arrow matrices from \eqref{eq:arrow_matrix_defs}, their
Shor linearizations are $\langle B_i,Y\rangle=0$, where
\begin{equation}\label{eq:sturm_matrix_defs}
B_i:=E_{ii}-E_{0,i+1}
\quad (i=1,\ldots,n-1),
\qquad B_n:=E_{nn}.
\end{equation}
The resulting Shor relaxation is
\begin{equation}\label{eq:binary_sturm_relaxation}
\left\{
Y\in\mathbb{S}^{n+1}_+
\ \middle|\
\langle A_0,Y\rangle=1,\quad
\langle A_i,Y\rangle=0\ (i=1,\ldots,n),\quad
\langle B_i,Y\rangle=0\ (i=1,\ldots,n)
\right\}.
\end{equation}

The arrow matrices and the matrices in \eqref{eq:sturm_matrix_defs} satisfy
\begin{equation}\label{eq:binary_sturm_diag_combo}
\sum_{i=1}^nE_{ii}
=
nB_n+\sum_{i=1}^{n-1}iB_i
-\sum_{i=2}^n(i-1)A_i .
\end{equation}
Consequently,
\[
W:=\operatorname{Diag}(0,1,\ldots,1)=\sum_{i=1}^nE_{ii}
\]
is a nonzero positive semidefinite matrix in the span of the
zero-right-hand-side constraint matrices, and hence is a valid exposing
matrix.  The face it exposes is
\[
F_1:=\mathbb{S}^{n+1}_+\cap W^\perp
=\{\alpha e_0e_0^T\mid \alpha\geq0\}.
\]
Indeed, a zero diagonal entry in a positive semidefinite matrix forces the
corresponding row and column to vanish.  Since $e_0e_0^T$ is feasible,
$F_1$ is the minimal face containing the feasible set; the normalization
constraint further shows that the feasible set is the singleton
$\{e_0e_0^T\}$.  Thus, one facial-reduction step reaches the minimal face, and
\eqref{eq:binary_sturm_relaxation} has singularity degree one for every
$n\geq1$.  Identity \eqref{eq:binary_sturm_diag_combo} shows how the arrow
constraints combine with the linearized quadratic equations to expose all
nonconstant coordinates at once.

\subsection{A construction with singularity degree
\texorpdfstring{$n$}{n}}
\label{sec:binary_shor_high_degree}

We now construct, for every $n\geq1$, a binary set defined by quadratic
equalities whose Shor relaxation has singularity degree $n$.  The construction
imposes Sturm's chain on nonsingular linear forms of the
binary variables.  An ordinary Vandermonde transformation then produces a
Hankel structure that allows the exposing matrices to be characterized
explicitly.

For $s=1,\ldots,n$ and $t>0$, set
\[
c_s(t):=(t,t^2,\ldots,t^s)^T.
\]
Fix distinct positive numbers $t_1,\ldots,t_n$, and define
\begin{equation}\label{eq:binary_high_degree_C}
C_{ij}:=t_i^j,
\qquad i,j=1,\ldots,n.
\end{equation}
Thus, row $i$ of $C$ is $c_n(t_i)^T$.  Factoring $t_i$ from each row yields
the standard Vandermonde matrix with row $i$ equal to
$(1,t_i,\ldots,t_i^{n-1})$.  Since the $t_i$ are distinct, $C$ is
nonsingular.  Let $D:=C^{-1}$, let
$d_i^T$ be row $i$ of $D$, and define
\[
\ell_i(x):=d_i^Tx,
\qquad i=1,\ldots,n.
\]
Using these linear forms, define a quadratically constrained binary set by the
following quadratic equalities in the original variables $x$:
\begin{equation}\label{eq:binary_high_degree_set}
\mathcal X_n
:=
\left\{
x\in\{0,1\}^n
\ \middle|\
\ell_i(x)^2-\ell_{i+1}(x)=0\ (i=1,\ldots,n-1),\ 
\ell_n(x)^2=0
\right\}.
\end{equation}
The last equation gives $\ell_n(x)=0$, and the remaining equations then give
$\ell_{n-1}(x)=\cdots=\ell_1(x)=0$.  Since $D$ is nonsingular,
$\mathcal X_n=\{0\}$.

The quadratic description in \eqref{eq:binary_high_degree_set} is essential.
The same binary set can be described by the linear equations $x=0$, whose
Shor relaxation has singularity degree one by
\Cref{thm:shor_linear_sd_bound}.  We now show that the displayed quadratic
description produces singularity degree $n$.

Writing $\tilde d_i:=\binom{0}{d_i}$, the Shor linearizations of the quadratic
equations are $\langle G_i,Y\rangle=0$, where
\[
G_i:=\tilde d_i\tilde d_i^T
-\frac12(e_0\tilde d_{i+1}^T+\tilde d_{i+1}e_0^T)
\quad(i=1,\ldots,n-1),
\qquad
G_n:=\tilde d_n\tilde d_n^T.
\]
The Shor relaxation associated with this quadratic description is
\begin{equation}\label{eq:binary_high_degree_shor}
\mathcal S_n
:=
\left\{
Y\in\mathbb S^{n+1}_+
\ \middle|\
\langle A_0,Y\rangle=1,\ 
\langle A_k,Y\rangle=0\ (k=1,\ldots,n),\ 
\langle G_i,Y\rangle=0\ (i=1,\ldots,n)
\right\}.
\end{equation}

To analyze $\mathcal S_n$, we introduce a family of lower-dimensional
spectrahedra.  For $s=1,\ldots,n$ and $k=1,\ldots,n$, define
$\widetilde A_k^{(s)}\in\mathbb S^{s+1}$ by
\[
\widetilde A_k^{(s)}
:=
\begin{pmatrix}
0&-\frac12c_s(t_k)^T\\[1mm]
-\frac12c_s(t_k)&c_s(t_k)c_s(t_k)^T
\end{pmatrix},
\]
and, using the symmetric basis matrices $E_{ij}\in\mathbb S^{s+1}$, define
$B_i^{(s)}\in\mathbb S^{s+1}$ by
\[
B_i^{(s)}:=E_{ii}-E_{0,i+1}
\quad(i=1,\ldots,s-1),
\qquad
B_s^{(s)}:=E_{ss}.
\]
Set
\begin{equation}\label{eq:rotated_spectrahedron_family}
\mathcal T_s
:=
\left\{
Z\in\mathbb S^{s+1}_+
\ \middle|\
Z_{00}=1,\ 
\langle\widetilde A_k^{(s)},Z\rangle=0\ (k=1,\ldots,n),\ 
\langle B_i^{(s)},Z\rangle=0\ (i=1,\ldots,s)
\right\}.
\end{equation}

\begin{thm}\label{thm:binary_shor_high_degree}
For every $n\geq1$, the set $\mathcal X_n$ in
\eqref{eq:binary_high_degree_set} is the singleton $\{0\}$, and its Shor
relaxation $\mathcal S_n$ has singularity degree $n$.
\end{thm}

\begin{proof}
We have already shown that $\mathcal X_n=\{0\}$.  We first relate
$\mathcal S_n$ to $\mathcal T_n$.  Since $D=C^{-1}$, $C^Td_i$ is the $i$th
standard unit vector of $\mathbb R^n$.  Hence, for
\[
R:=\begin{pmatrix}1&0\\0&C\end{pmatrix},
\]
a direct block calculation, detailed in
\Cref{app:congruence_derivation}, gives
\[
R^TG_iR=B_i^{(n)}\quad(i=1,\ldots,n),
\qquad
R^TA_kR=\widetilde A_k^{(n)}\quad(k=1,\ldots,n),
\qquad
R^TA_0R=A_0.
\]
Consequently, $Z\mapsto RZR^T$ maps $\mathcal T_n$ bijectively onto
$\mathcal S_n$.  Since $R$ is nonsingular, this congruence preserves
facial-reduction sequences and singularity degree.  It therefore remains to
prove that $\sd(\mathcal T_n)=n$.

We prove the stronger statement
\[
\sd(\mathcal T_s)=s,
\qquad s=1,\ldots,n.
\]
Fix $s$.  In the first facial-reduction auxiliary system for $\mathcal T_s$,
the multiplier of the normalization constraint is zero because it is the only
constraint with a nonzero right-hand side.  Thus, every first-step exposing
matrix has the form
\begin{equation}\label{eq:consecutive_exposing_matrix}
W
=
\sum_{i=1}^s\lambda_iB_i^{(s)}
+\sum_{k=1}^n\mu_k\widetilde A_k^{(s)}
\succeq0.
\end{equation}
Define
\[
a_r:=\sum_{k=1}^n\mu_k t_k^r,
\qquad r=1,\ldots,2s.
\]
Here $r$ indexes the exponent, and each $a_r$ is a sum over
$k=1,\ldots,n$.  Since component $j$ of $c_s(t_k)$ is $t_k^j$, the
contribution from the matrices $\widetilde A_k^{(s)}$ is
\begin{equation}\label{eq:consecutive_moment_block}
\sum_{k=1}^n\mu_k\widetilde A_k^{(s)}
=
\begin{pmatrix}
0&-\frac12(a_1,\ldots,a_s)\\[1mm]
-\frac12(a_1,\ldots,a_s)^T&(a_{i+j})_{i,j=1}^s
\end{pmatrix}.
\end{equation}
Thus, row $0$ involves only $a_1,\ldots,a_s$, whereas the lower-right
block involves moments through $a_{2s}$.
With the convention $\lambda_0:=0$, the contribution from the matrices
$B_i^{(s)}$ is
\begin{equation}\label{eq:consecutive_B_block}
\sum_{i=1}^s\lambda_iB_i^{(s)}
=
\begin{pmatrix}
0&-\frac12(\lambda_0,\lambda_1,\ldots,\lambda_{s-1})\\[1mm]
-\frac12(\lambda_0,\lambda_1,\ldots,\lambda_{s-1})^T
&\Diag(\lambda_1,\ldots,\lambda_s)
\end{pmatrix}.
\end{equation}
Every matrix in \eqref{eq:consecutive_exposing_matrix} has zero $(0,0)$
entry.  Positive semidefiniteness therefore forces row and column $0$ of $W$
to vanish.  Combining \eqref{eq:consecutive_moment_block} and
\eqref{eq:consecutive_B_block}, the entries of row $0$ give
\begin{equation}\label{eq:consecutive_zeroth_row}
a_1=0,
\qquad
\lambda_i=-a_{i+1}\quad(i=1,\ldots,s-1).
\end{equation}
In particular, the lower-right block in
\eqref{eq:consecutive_moment_block} is Hankel, with $(i,j)$ entry
\begin{equation}\label{eq:consecutive_hankel_entries}
a_{i+j},
\qquad i,j=1,\ldots,s.
\end{equation}

We claim that every row and column of $W$, except possibly the one indexed by
$s$, is zero.  The claim is immediate for $s=1$.  Suppose $s\geq2$.  By
\eqref{eq:consecutive_zeroth_row} and
\eqref{eq:consecutive_hankel_entries},
\[
W_{11}=\lambda_1+a_2=-a_2+a_2=0.
\]
Since $W\succeq0$, row and column $1$ must vanish.  Their entries indexed by
$j=2,\ldots,s$ satisfy $W_{1j}=a_{1+j}$, and hence
\[
a_3=a_4=\cdots=a_{s+1}=0.
\]
For the induction step, let $2\leq r\leq s-1$ and suppose that rows and columns
$1,\ldots,r-1$ vanish and
\[
a_3=a_4=\cdots=a_{s+r-1}=0.
\]
Both $r+1$ and $2r$ lie between $3$ and $s+r-1$.  Therefore
\[
W_{rr}
=\lambda_r+a_{2r}
=-a_{r+1}+a_{2r}
=0.
\]
Positive semidefiniteness again forces row and column $r$ to vanish.  In
particular, $W_{rs}=a_{s+r}=0$, which extends the zero sequence by one term.
Induction therefore gives $a_3=\cdots=a_{2s-1}=0$ and shows that rows and
columns $1,\ldots,s-1$ vanish.  Hence
$W=\gamma E_{ss}$ for some $\gamma\geq0$.  Since
$E_{ss}=B_s^{(s)}$, every positive multiple of $E_{ss}$ is a valid exposing
matrix.

Therefore every nonzero first-step exposing matrix for $\mathcal T_s$ exposes
the face on which row and column $s$ vanish.  For $s\geq2$, deleting that row
and column truncates $c_s(t)$ to $c_{s-1}(t)$ and transforms the remaining
constraints exactly into those defining $\mathcal T_{s-1}$.
More explicitly, $B_i^{(s)}$ restricts to
$B_i^{(s-1)}$ for $i=1,\ldots,s-1$, with the $E_{0s}$ term in
$B_{s-1}^{(s)}$ disappearing, while $B_s^{(s)}$ restricts to zero.
Thus, the
reduced spectrahedron is $\mathcal T_{s-1}$.  For $s=1$,
the same reduction leaves $\mathcal T_0:=\{1\}$, which is strictly feasible.
Hence $\sd(\mathcal T_0)=0$, and induction gives
\[
\sd(\mathcal T_s)
=1+\sd(\mathcal T_{s-1})
=s,
\qquad s=1,\ldots,n.
\]
Since $\mathcal S_n$ and $\mathcal T_n$ are congruent,
$\sd(\mathcal S_n)=n$.
\end{proof}

The direct formulation \eqref{eq:boolean_sturm_attempt} and the rotated
formulation \eqref{eq:binary_high_degree_set} both define the singleton
$\{0\}$, but their Shor relaxations have different singularity degrees.  In
the direct formulation, the arrow and quadratic constraint matrices combine
to expose all nonconstant coordinates in one step.  In the rotated
formulation, the argument above shows that every step can expose only the
last remaining coordinate.  This is precisely the
formulation dependence underlying the contrast between singularity degree one
and singularity degree $n$.

\section{Conclusion}

We have established a sharp distinction between Shor relaxations of linearly
and quadratically constrained binary sets.  For every nonempty linearly
constrained binary set, the Shor relaxation has singularity degree
at most one.  Its singularity degree is zero precisely when the LP relaxation
contains a point in the open cube that strictly satisfies every inequality;
otherwise, a single facial-reduction step reaches the minimal face.  In
contrast, in every dimension we constructed a binary set defined by quadratic
equalities whose Shor relaxation has a positive semidefinite matrix variable
of order one greater than the dimension and singularity degree equal to the
dimension.

From the perspective of facial regularity, the result for linearly constrained
binary sets is reassuring.  Any failure of Slater's condition can be resolved
in a single facial-reduction step, and the minimal face can be identified by
solving an LP and applying explicit linear algebra, rather than by solving an
auxiliary SDP.  Moreover, even if facial reduction is not performed, the
singularity-degree bound of one rules out the deeper degeneracy and
progressively weaker error bounds associated with long facial-reduction
sequences.  Thus, the result
provides a favorable regularity guarantee for the numerical use of Shor
relaxations.  No such uniform guarantee holds for quadratic descriptions:
their singularity degree can grow linearly with the order of the lifted
matrix.
\appendix

\section{Details of the congruence calculation}
\label{app:congruence_derivation}

We verify the congruence identities used in the proof of
\Cref{thm:binary_shor_high_degree}.  Let $\hat e_i$ denote the $i$th standard
unit vector of $\mathbb R^n$.  Since $D=C^{-1}$ and $d_i^T$ is row $i$ of
$D$,
\[
C^Td_i=\hat e_i,
\qquad i=1,\ldots,n.
\]
Therefore, for
$R=\left(\begin{smallmatrix}1&0\\0&C\end{smallmatrix}\right)$ and
$\tilde d_i=\binom{0}{d_i}$,
\begin{equation}\label{eq:appendix_transformed_vectors}
R^T\tilde d_i=\binom{0}{C^Td_i}=\binom{0}{\hat e_i}=e_i,
\qquad R^Te_0=e_0.
\end{equation}

For $i=1,\ldots,n-1$, substituting
\eqref{eq:appendix_transformed_vectors} into the definition of $G_i$ gives
\begin{align*}
R^TG_iR
&=(R^T\tilde d_i)(R^T\tilde d_i)^T
-\frac12\bigl((R^Te_0)(R^T\tilde d_{i+1})^T
+(R^T\tilde d_{i+1})(R^Te_0)^T\bigr)\\
&=e_ie_i^T-\frac12(e_0e_{i+1}^T+e_{i+1}e_0^T)\\
&=E_{ii}-E_{0,i+1}=B_i^{(n)}.
\end{align*}
For the terminal equation,
\[
R^TG_nR=(R^T\tilde d_n)(R^T\tilde d_n)^T
=e_ne_n^T=E_{nn}=B_n^{(n)}.
\]

It remains to transform the arrow and normalization matrices.  For
$k=1,\ldots,n$, the block form of $A_k=E_{kk}-E_{0k}$ is
\[
A_k=
\begin{pmatrix}
0&-\frac12\hat e_k^T\\[1mm]
-\frac12\hat e_k&\hat e_k\hat e_k^T
\end{pmatrix},
\]
and hence
\begin{align*}
R^TA_kR
&=
\begin{pmatrix}
0&-\frac12\hat e_k^TC\\[1mm]
-\frac12C^T\hat e_k&C^T\hat e_k\hat e_k^TC
\end{pmatrix}\\
&=
\begin{pmatrix}
0&-\frac12c_n(t_k)^T\\[1mm]
-\frac12c_n(t_k)&c_n(t_k)c_n(t_k)^T
\end{pmatrix}
=\widetilde A_k^{(n)},
\end{align*}
because $\hat e_k^TC=c_n(t_k)^T$ is row $k$ of $C$.  Finally,
\[
R^TA_0R=R^TE_{00}R=E_{00}=A_0.
\]
These calculations establish all three congruence identities used in the
proof.

\bibliographystyle{siam}
\bibliography{mybib}
\addcontentsline{toc}{section}{Bibliography}

\end{document}